\documentclass[]{interact}

\usepackage{epstopdf}
\usepackage{subfigure}
\usepackage{float}
\usepackage[longnamesfirst,sort]{natbib}
\bibpunct[, ]{(}{)}{;}{a}{,}{,}

\theoremstyle{plain}
\newtheorem{theorem}{Theorem}[section]
\newtheorem{assumption}{Assumption}

\newtheorem{proposition}[theorem]{Proposition}

\theoremstyle{definition}

\theoremstyle{remark}
\newtheorem{remark}{Remark}

\def\tr{\mbox{tr}}
\def\min{\mbox{min}}
\def\max{\mbox{max}}
\def\rea{\mathbb{R}}

\def\cale{{\mathcal E}}

\def\begequarrs{\begin{eqnarray*}}
	\def\endequarrs{\end{eqnarray*}}
\def\begequarr{\begin{eqnarray}}
\def\endequarr{\end{eqnarray}}
\def\begarr{\begin{array}}
	\def\endarr{\end{array}}
\def\begequ{\begin{equation}}
\def\endequ{\end{equation}}
\def\lab{\label}
\def\begdes{\begin{description}}
	\def\enddes{\end{description}}
\def\begenu{\begin{enumerate}}
	\def\begite{\begin{itemize}}
		\def\endite{\end{itemize}}
	\def\endenu{\end{enumerate}}

\def\lef[{\left[\begin{array}}
	\def\rig]{\end{array}\right]}
\def\qed{\hfill$\Box \Box \Box$}
\def\begcen{\begin{center}}
	\def\endcen{\end{center}}
\def\begrem{\begin{remark}\rm}
	\def\endrem{\end{remark}}
\def\begcas{\begin{cases}}
	\def\endcas{\end{cases}}

\usepackage{color}

\begin{document}

\articletype{}

\title{An Adaptive Passivity-Based Controller of a Buck-Boost Converter With a Constant Power Load}

\author{
\name{Wei He\textsuperscript{a}\thanks{CONTACT Shihua Li. Email: lsh@seu.edu.cn}, Romeo Ortega\textsuperscript{b}, Juan E. Machado\textsuperscript{b}, Shihua Li\textsuperscript{a}}
\affil{\textsuperscript{a}Key Laboratory of Measurement and Control of Complex Systems of Engineering,
Ministry of Education, School of Automation, Southeast University, 210096 Nanjing, China; \textsuperscript{b}Laboratoire des Signaux et Systemes, Supelec, Plateau du Moulon, 91192 Gif-sur-Yvette, France.}
}

\maketitle

\begin{abstract}
\noindent This paper addresses the problem of regulating the output voltage of a DC-DC buck-boost converter feeding a constant power load, which is a problem of current practical interest. Designing a stabilising controller is theoretically challenging because its average model is a bilinear second order system that, due to the presence of the constant power load, is non-minimum phase with respect to both states.  Moreover, to design a high-performance controller, the knowledge of the extracted load power, which is difficult to measure in industrial applications, is required. In this paper, an adaptive interconnection and damping assignment passivity-based control---that incorporates the immersion and invariance parameter estimator for the load power---is proposed to solve the problem. Some detailed simulations are provided to validate the transient behaviour of the proposed controller and compare it with the performance of a classical PD scheme.
\end{abstract}

\begin{keywords}
Buck-boost converter; constant power load; interconnection and damping assignment passivity based control; immersion and invariance
\end{keywords}

\section{Introduction}
The DC-DC buck-boost power converter is increasingly utilized in power distribution systems since it can step up or down the voltage between the source and load, providing flexibility in choosing the voltage rating of the DC source. Although the control of these converters in the face of classical loads is well-understood, in some modern applications the loads  do not behave like standard passive impedances, instead they are more accurately represented as constant power loads (CPLs), which correspond to first-third quadrant hyperbolas in the loads voltage-current plane. This scenario significantly differs from the classical one and poses a new challenge to control theorist, see \citep{Bar16,Ema06,Kha08,Mar12} for further discussion on the topic and \citep{Sin17} for a recent review of the literature. It should be underscored that the typical application of this device requires large variations of the operating point---therefore, the dynamic description of its behavior cannot be captured by a linearized model, requiring instead a nonlinear one.

To the best of the authors' knowledge no controller, with guaranteed stability properties for the nonlinear model, has been proposed for the voltage regulation of the buck-boost converter with a CPL---hence, its solution remains an open problem. Several techniques to address this problem, but without a nonlinear stability analysis, have been reported in the power electronics literature. In \citep{Rah09}, the active-damping approach is utilized to address the negative impedance instability problem raised by the CPL. The main idea of this method is that a virtual resistance is considered in the original circuit to increase the system damping. However, the stability result is obtained by applying small-signal analysis, which is valid only in a small neighbourhood of the operating point. A new nonlinear feedback controller, which is called ``Loop Cancellation", has been proposed to stabilize the buck-boost converter by ``cancelling the destabilizing behaviour caused by CPL" \citep{Rah10}. The control problem turns into the design of a controller for the linear system by using loop cancellation method. However, the construction is based on feedback linearization \citep{ISI95} that, as is well-known, is highly non-robust. A sliding mode controller is designed in \citep{Sin16} for this problem. However, for the considered nonlinear system, the stability result is obtained by adopting the linear system theory. In addition, as it is widely acknowledged, the drawbacks of this method are that the proposed control law suffers from chattering and its relay action injects a high switching gain. The deleterious effect of these factors is clearly illustrated in experiments shown in \citep{Sin16}, which exhibit a very poor performance.

Aware of the need to deal with the intrinsic nonlinearities of power converters some authors of this community have applied passivity-based controllers (PBCs), which is a natural candidate in these applications. Unfortunately, in many of these reports the theoretical requirements of the PBC methodology are not rigorously respected. In \citep{Wei} a review of some of the---alas, incorrect or incomplete---results on application of PBC for power converters is given. For instance, in \citep{Kwa07}, the well-known standard  PBC \citep{Ort98} is used for the buck-boost with a CPL. Unfortunately, the given result is theoretically incorrect due to the fact that the authors fail to validate the stability of the zero dynamics of the system with respect to the controlled output that, as explained in \citep{Ort98}, is an essential step for the stability analysis and, as shown in this paper, it turns out to be violated.

An additional drawback of the existing results is that all of them require the knowledge of the power extracted by the CPL, which is difficult to measure in industrial applications. Designing an estimator for the power is a hard task because the original system is nonlinear and the only available measurements are inductor current and output voltage.

In this paper, we apply the well-known interconnection and damping assignment (IDA) PBC, first reported in \citep{Ort02} and reviewed in \citep{Ort04}, to stabilize the buck-boost converter with a CPL. The main contributions of this note are:
\begite
\item[1)] Derivation of an IDA-PBC that ensures the desired operating point is a (locally) asymptotically stable equilibrium with a guaranteed domain of attraction.
\item[2)]  Design of an estimator of load power, which is based on the immersion and invariance (I$\&$I) technique \citep{Ast08} and has guaranteed stability properties, to make the IDA-PBC adaptive.
\item[3)] Proof that the zero dynamics of the system, with respect to both states, is unstable---limiting the achievable performance of classical PD controllers to ``low-gain tunings".
\endite

The remaining of the paper is organized as follows. Section \ref{section2} contains the model of the system  and the analysis of its zero dynamics with respect to the two states. Moreover, a remark on the  result reported in \citep{Kwa07} is given. Section \ref{section3} proposes the IDA-PBC assuming the power is known. To make the latter scheme adaptive in Section \ref{section4} we design a on-line power estimator, while some simulations carried out by MATLAB are provided in Section \ref{section5}. This paper is wrapped-up with some concluding remarks in Section \ref{section6}. To enhance readability, the derivation of the IDA-PBC that is conceptually simple but computationally involved, is given in an appendix at the end of the paper.

\section{System Model, Problem Formation and Zero Dynamics Analysis}
\label{section2}
In this section, the average model of buck-boost converter feeding a CPL, its zero dynamics analysis and a remark on the existing result on \citep{Kwa07} are given.
 \subsection{Model of buck-boost converter with CPL}
\label{subsec21}

The topology of buck-boost converter feeding a CPL, is shown in Fig. \ref{buckboostcircuit}. Under the standard assumption that it operates in continuous conduction mode, the average model is given by
\begin{align}
\nonumber
L {di \over dt}&= -(1-u)v + uE, \\
C {dv \over dt}&= (1-u)i-\frac{P}{v},
\lab{sys}
\end{align}
where $i\in \mathbb{R}_{>0}$ is the inductor current, $v\in \mathbb{R}_{>0}$ the output voltage, $P\in \mathbb{R}_{>0}$ the power extracted by the CPL, $E\in \mathbb{R}_{>0}$ is the input voltage and $u\in [0,1]$ is the duty ratio, which is the control signal.

\begin{figure}[H]
  \centering\includegraphics[scale=1]{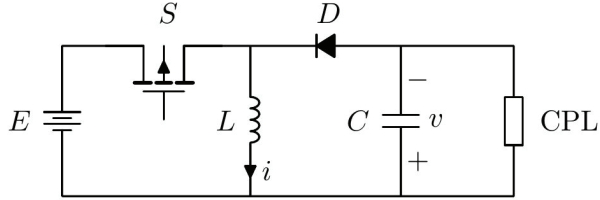}
  \caption{\normalsize Circuit representation of the DC-DC buck-boost converter with a CPL}\label{buckboostcircuit}
\end{figure}
Some simple calculations show that the assignable equilibrium set is given by
\begin{align}
\label{set}
\mathcal{E}:=\left\{ (i,v) \in \mathbb{R}^2_{>0}\;|\; i- P\left(\frac{1}{v}+ \frac{1}{E}\right)=0  \right\}.
\end{align}

\subsection{Control problem formulation}
\label{subsec22}
Consider the system \eqref{sys} verifying the following conditions.
\begin{assumption}\em
\lab{ass1}
The power load $P$ is {\em unknown} but the parameters $L,C$ and $E$ are known.
\end{assumption}

\begin{assumption}\em
\lab{ass2}
The state $(i,v)$ is measurable.
\end{assumption}

Fix a {\em desired output voltage} $v_\star \in \rea_{>0}$ and compute the associated assignable equilibrium point $(i_\star,v_\star) \in \cale$. Design a static state-feedback control law with the following features.
\begite
\item[(F1)]  $(i_\star,v_\star)$  is an asymptotically stable equilibrium of the closed-loop with a well-defined domain of attraction.
\item[(F2)] It is possible to define a set $\Omega \subset \rea^2_{>0}$ which is {\em invariant} and inside the domain of attraction of the equilibrium. That is, a set inside the positive orthant verifying
\begequarrs
&&\left[(i(0),v(0)) \in \Omega  \;\Rightarrow\;  (i(t),v(t)) \in \Omega,\forall t\geq 0 \right]\\
&& \lim_{t \to \infty}(i(t),v(t))  =  (i_\star,v_\star).
\endequarrs
\endite

To simplify the notation, and without loss of generality, in the sequel we consider the normalized model of the system, which is obtained using the change of coordinates
\begequarr
\nonumber
x_1 &:= &\frac{1}{E}\sqrt{\frac{L}{C}}i \\
x_2 & := & \frac{1}{E}v,
\lab{chacoo}
\endequarr
and doing the {\em time scale} change  $\tau=\frac{t}{\sqrt{LC}}$ that yields the model
\begin{align}
\dot x_1&=-(1-u)x_2+u
\nonumber \\
\dot x_2&=(1-u)x_1-\frac{D}{x_2}
\label{bucboo}
\end{align}
where
$$
D:=\frac{P}{E^2}\sqrt{\frac{L}{C}},
$$
$\dot{(\cdot)}$ denotes ${d \over d\tau}(\cdot)$ and all signals are expressed in the new time scale $\tau$. The assignable equilibrium set $\cale$ in the coordinates $x$ is given by
\begin{align}\label{set}
\mathcal{E}_x:=\left\{ x \in \mathbb{R}^2_{>0}\;|\; x_1-\frac{D}{x_2}-D=0\right\}
\end{align}

Notice that, under Assumptions \ref{ass1} and \ref{ass2}, the control problem is translated into the design of a state feedback for the system \eqref{bucboo}, with unknown $D$, such that a given $x_\star \in \cale_x$ is asymptotically stable.

It is important to recall that the signal of interest is the output voltage $v$, therefore, for the fixed $x_{2\star} \in \rea_{>0}$, the $x_{1\star} \in \rea_{>0}$ is defined via
\begequ
\lab{x1sta}
x_{1\star}=\frac{D}{x_{2\star}}+D.
\endequ

\subsection{Stability analysis of the systems zero dynamics}
\label{subsec23}
The design of a stabilising controller for \eqref{bucboo} is complicated by the fact that, as shown in the proposition below, its zero dynamics with respect to both states is {\em unstable}. This means that, if the controller injects high gain, the closed-loop system will be {\em unstable}---as it stems from the fact that the poles will move towards the unstable zeros.  This situation hampers the design of high performance PD controllers, which require high proportional gains to speed up the transients. See Section \ref{section5} for an illustration of this fact.
\begin{proposition}\label{lemma1} \em
Consider the system \eqref{bucboo} and an assignable equilibrium $x_\star \in \cale_x$. The zero dynamics with respect to the outputs  $x_1-x_{1\star}$ or $x_2-x_{2\star}$ are unstable.
\end{proposition}

\begin{proof}
Fixing  $x_1=x_{1\star}$ and using the first equation in \eqref{bucboo} we get
\begin{align*}
u=\frac{x_2}{x_2+1},
\end{align*}
which substituted in the second equation of \eqref{bucboo} yields the zero dynamics
\begin{align}
\lab{zerdynx1}
\dot x_2=\frac{D}{x_{2\star}x_2(x_2+1)}(x_{2}-x_{2\star})=:s(x_2).
\end{align}
The slope of $s(x_2)$ evaluated at  $x_2=x_{2\star}$ gives
$$
s'(x_2)|_{ x_2=x_{2\star}}= \frac{D}{x_{2\star}^2(1+x_{2\star})}.
$$
Since $x_{2\star}>0$, this is a positive number proving that the equilibrium $x_{2\star}$ of the dynamics \eqref{zerdynx1} is unstable---as claimed by the proposition.

Now, fixing $x_2=x_{2\star}$ and using the second equation of \eqref{bucboo} we get
\begin{align}
u=1-\frac{D}{x_1x_{2\star}},
\end{align}
which substituted in the first equation of \eqref{bucboo} yields
\begin{align}
\dot x_1 = 1- \frac{x_{1\star}-D}{x_1}=:w(x_1).
\end{align}
Proceeding as done for the case above we get
$$
w'(x_1)|_{ x_1=x_{1\star}}= \frac{x_{1\star}-D}{x^2_{1\star}}.
$$
The proof is completed noting from \eqref{set} that $x_{1\star}=D(1+\frac{1}{x_{2\star}})>D$ and the slope is, again, positive.
\end{proof}
The proof that the zero dynamic of \eqref{bucboo} with respect to $x_1 - x_{1\star}$ is unstable invalidates the stability claim made in Section \uppercase\expandafter{\romannumeral4} of \citep{Kwa07}. In this paper, a standard PBC is designed fixing $x_1=x_{1\star}$. It is well-known \citep{Ort98} that this kind of controller implements an inversion of the systems zero dynamics, therefore the controller will be unstable if the zero dynamics is unstable, which is the case of the PBC of \citep{Kwa07}. The interested reader is referred to \citep{Wei} for further details on this problem and a discussion of a similar---unfortunate---situation in other papers of PBC for power converters reported in the literature.
\section{Interconnection and Damping Assignment Passivity-based Controller}
\label{section3}
In this section, the IDA-PBC approach is proposed to stabilize the buck-boost converter feeding a CPL by assuming the power $D$ is known. This condition is later relaxed in Proposition \ref{proposition2} where an estimator of $D$ is added to the IDA-PBC.

To make the paper self-contained we present below the main result of the IDA-PBC methodology and give its proof. For more details on IDA-PBC  we refer the reader to \citep{Ort04}.
\begin{proposition}
\label{proposition0}\em
Consider the nonlinear system
\begin{align}
\lab{dotxfg}
\dot x=f(x)+g(x)u
\end{align}
with state $x \in \rea^n$ and control $u \in \rea^m$ and a desired operating point
$$
x_\star \in \{x \in \rea^n\;|\; g^\perp(x)f(x)=0\},
$$
where $g^\perp(x)$ is a full rank left annihilator of $g(x)$. Fix the target dynamics as
\begequ
\lab{tardyn}
\dot x = F_d(x) \nabla H_d(x),
\endequ
where $\nabla H_d(x):=\left({\partial H_d(x) \over \partial x}\right)^\top $ with the function $H_d(x)$ a solution of the PDE
\begequ
\label{pde}
  g^\perp(x)\left[f(x)-F_d(x)\nabla H_d(x)\right]=0,
\endequ
verifying
\begequ
\lab{mincon}
x_\star = \arg \min \{H_d(x)\},
\endequ
and the matrix $F_d(x)$ is such that
\begin{align}
\label{fd}
F^\top _d(x)+F_d(x)< 0.
\end{align}

The system \eqref{dotxfg} in closed-loop with
\begequ
\lab{uida}
u = \bar u(x):= [g^\top (x)g(x)]^{-1}g^\top (x)[ F_d(x) \nabla H_d(x)-f(x)]
\endequ
has an asymptotically stable equilibrium at $x_\star$ with {\em strict Lyapunov function} $H_d(x)$
\end{proposition}

\begin{proof}
From the fact that the $n \times n$ matrix $\lef[{c} g^\top (x) \\ g^\perp(x)\rig]$ is full rank we have the following equivalence
$$
f(x)+g(x)\bar u(x)=F_d(x)\nabla H_d(x)\;\Leftrightarrow\;\eqref{pde},\;\eqref{uida}.
$$
Hence, the closed-loop is given by \eqref{tardyn}. Now, \eqref{mincon} ensures $H_d(x)$ is positive definite (with respect to $x_\star$). Computing the derivative of $H_d(x)$ along the trajectories of \eqref{tardyn} and invoking \eqref{fd} we get
$$
\dot H_d = (\nabla H_d(x))^\top F_d(x) \nabla H_d(x) < 0,\; \forall x \neq x_\star
$$
and, therefore, $H_d(x)$ is a {strict Lyapunov function} for the closed-loop system completing the proof.
\end{proof}

The proposition below is a direct application of IDA-PBC that provides a solution to our problem.
\begin{proposition}
\label{proposition1}\em
Consider the average model of the DC-DC buck-boost converter with a CPL \eqref{bucboo} with $D$ {\em known} and satisfying Assumption \ref{ass2}. Fix  $x_{2\star} \in \rea_{>0}$ and  compute $x_{1\star} \in \rea_{>0}$ via \eqref{x1sta}. The IDA-PBC
\begin{align}
\lab{ubucboo}
u =& \bar u(x,D,k_1):=\frac{1}{x_1^2+(x_2+1)^2} \Bigg(x_2(x_2+1)+x_1(x_1-\frac{D}{x_2})-\Big(\frac{x_2(x_2+1)}{x_1}+\frac{2x_1x_2}{x_2+1}\Big)\nonumber\\
&\Big(k_1x_1\big(2(k_2+x_1^2)+x_2^2-\frac{D(1+x_2)}{2x_1^2+x_2^2}\big)+\frac{\sqrt{2} D x_{1} \arctan{\frac{x_{1}} {\sqrt{x_{1}^2+\frac{x_{2}^2}{2}}}}}{(2x_{1}^2+x_{2}^2)^{\frac{3}{2}}}\Big)+\frac{1}{2x_2(2x_1^2+x_2^2)^{\frac{3}{2}}}\nonumber\\
&\Big(\frac{2x_1^2}{(x_2+1)^2}-2x_2\Big)\Big(\sqrt{2x_1^2+x_2^2}\big(2Dx_1(1+x_2)
+x_2(2x_1^2+x_2^2)(-1+2k_1x_2(k_2+x_1^2)\nonumber\\
&+k_1x_2^3)\big)+\sqrt{2}Dx_2^2\arctan{\frac{x_1}{\sqrt{x_1^2+\frac{x_2^2}{2}}}}\Big)\Bigg),
\end{align}
where  $k_1$ is a tuning gain satisfying
\begequ
\lab{bouk1}
k_1  > \max \{k_1', k_1''\},
\endequ
where the constants $k_1',k_1''$ are defined  in Appendix A in \eqref{kpri} and \eqref{kbipri}, respectively,  and the constant $k_2$ is defined as\footnote{Although $x_{1\star}$ is defined via \eqref{x1sta}, to simplify the notation, we have omitted this clarification in the definition of $k_2$.}
\begequ
\lab{k2}
k_2:={1 \over k_1}\left[\frac{D(1+x_{2\star})}{2x_{1\star}(2x_{1\star}^2+x_{2\star}^2)}-\frac{\sqrt{2} D x_{1\star} \arctan{\frac{x_{1\star}} {\sqrt{x_{1\star}^2
+\frac{x_{2\star}^2}{2}}}}}{2x_{1\star}(2x_{1\star}^2+x_{2\star}^2)^{\frac{3}{2}}}\right]-\frac{x_{2\star}^2}{2}-x_{1\star}^2,
\endequ
ensures the following.

\begdes
  \item[P1:] $x_\star$ is an asymptotically stable equilibrium of the closed-loop with Lyapunov function
\begequ
\label{solution} \hspace{-1cm}
H_d(x)=-\frac{1}{2}\left(x_2+\sqrt{2}D\arctan{\left[\frac{\sqrt{2}x_1}{x_2}\right]}\right)
-\frac{D\arctan{\left[\frac{x_1}{\sqrt{x_1^2+\frac{x_2^2}{2}}}\right]}}{2\sqrt{x_1^2+\frac{x_2^2}{2}}}
+\frac{k_1}{2}(x_1^2+\frac{x_2^2}{2}+k_2)^2.
\endequ
\enddes
\begdes
\item[P2:]  There exists a positive constant $c$ such that the sublevel set of the function $H_d(x)$
 \begin{align}\label{region}
  \Omega_x:=\{x \in \mathbb{R}_{>0}^{2}\;|\; H_d(x)\leq c\},
  \end{align}
is an estimate of the domain of attraction ensuring the state trajectories {\em remain} in $\mathbb{R}_{>0}^{2}$. That is, for all $x(0)\subset \Omega_x$, we have $x(t) \subset \Omega_x,\forall t \geq 0,$ and $\lim_{t \to \infty}x(t)=x_\star$.
\enddes
\qed
\end{proposition}

The control \eqref{ubucboo} is, obviously, extremely complicated for a practical implementation. This can be carried out doing an approximation, {\em e.g.}, via polynomials or rational functions, of this function for which standard symbolic software is readily available.

\section{Adaptive IDA-PBC Using an Immersion and Invariance Power Estimator}
\label{section4}

In this section, the case of unknown power $D$ is considered and an estimator of this parameter, based on the  I$\&$I technique \citep{Ast08}, is presented.

\begin{proposition}\label{proposition2}\em
	Consider the average model of the DC-DC buck-boost converter with CPL \eqref{bucboo} satisfying Assumptions \ref{ass1} and \ref{ass2} in closed-loop with an adaptive version of the control \eqref{ubucboo} given as
\begequ
\lab{adacon}
u = \bar u(x,\hat D,k_1)
\endequ
where $\hat D(t)$ is an on-line estimate of $D$ generated with the I$\&$I estimator
	\begin{align}\lab{boest1}
	\hat D=& -\frac{1}{2}\gamma x_2^2+{ D}_I\\
	\dot{D}_I=&\gamma x_1x_2 (1-u)+ \frac{1}{2}\gamma^2x_2^2 -\gamma D_I \label{boest2}
	\end{align}
where $\gamma>0$ is a free gain. There exists $k_1^{\min}$ such that for all $k_1>k_1^{\min}$ the overall system has an {\em asymptotically stable} equilibrium at $(x,\hat D)=(x_\star,D)$ . Moreover, {\em for all} initial conditions of the closed-loop system and all $D_I(0)$, we have
    \begin{align}\label{estimatorerror}
       \tilde D(t)=e^{- \gamma t} \tilde D(0),
    \end{align}
where $\tilde D:=\hat D - D$ is the parameter estimation error.
\end{proposition}

\begin{proof}
Differentiating $\tilde D$ along the trajectories of \eqref{bucboo} and using \eqref{boest1} one gets
\begin{align}
\dot {\tilde D}
&= -\gamma x_2\dot x_2 + \dot {D}_I\nonumber\\
&= -\gamma x_1x_2(1-u) + \gamma D + \dot {D}_I.\nonumber
\end{align}
Substituting \eqref{boest2} in the last equation
yields
\begin{align}
\dot {\tilde D}&= \gamma D + \frac{1}{2}\gamma^2x_2^2 -\gamma D_I\nonumber\\
&=-\gamma \tilde D,\nonumber
\end{align}
from which \eqref{estimatorerror} follows immediately.

To prove asymptotic stability of $(x,\hat D)=(x_\star,D)$ we  write the adaptive controller \eqref{adacon} as
$$
\bar u(x,\hat D,k_1) = \bar u(x, D,k_1) + \delta(x, \tilde D,k_1),
$$
where we define the mapping
$$
\delta(x,\tilde D,k_1) := \bar u(x,\tilde D + D, k_1)- \bar u(x,D,k_1),
$$
and we underscore the fact that $\delta(x,0,k_1)=0.$

Invoking the proof of Proposition \ref{proposition1} the closed-loop system is now a cascaded system of the form
\begequarrs
\dot x & = & F_d(x) \nabla H_d(x)+g(x)\delta(x,\tilde D,k_1) \\
\dot {\tilde D}&= &-\gamma \tilde D,
\endequarrs
where
\begequ
\lab{g}
g(x):=\left(
  \begin{array}{c}
    x_2+1 \\
    -x_1 \\
  \end{array}
\right),
\endequ
is the systems input matrix. Now, $\tilde D(t)$ tends to zero exponentially fast for all initial conditions, and for sufficiently large $k_1$, {\em i.e.}, such that \eqref{bouk1} is satisfied, the system above with $\tilde D=0$ is asymptotically stable. Invoking well-known results of asymptotic stability of cascaded systems, {\em e.g.}, Proposition 4.1 of \citep{Sep97}, completes the proof of (local) asymptotic stability.
\end{proof}

\section{Simulation Results}
\lab{section5}

In this section the performance of the proposed adaptive IDA-PBC is illustrated via some computer simulations. Moreover, its transient behavior is compared with the one of a PD controller designed adopting the classical linearization technique.

In all simulations, we have chosen the system parameters given in \citep{Kwa07}, namely, $P=61.25 W,\;C=500 \mu F, \;L=470\mu H,\;E=10 V, $ and fixed the desired equilibrium as $x_{\star}=(0.7423,4).$  For simplicity, we simulate the scaled system \eqref{bucboo}, for which we have $D=0.59384$. Depending on the context, the plots are shown either for $x$ or for $(i,v)$---that we recall are simply related by the scaling factors given in \eqref{chacoo}.

\subsection{PD controller}
To underscore the limitations of the PD controller and the difficulties related with its tuning we present now the local stability analysis of such a controller. From \eqref{bucboo}, we obtain the error dynamics
\begin{align*}
\dot e_1&=-(1-e_u-u_\star)(e_2+x_{2\star})+e_u+u_{\star}\\
\dot e_2&=(1-e_u-u_\star)(e_1+x_{1\star})-\frac{D}{e_2+x_{2\star}}
\end{align*}
where we define the errors
$$
e_1:=x_1-x_{1\star}, \;e_2:=x_2-x_{2\star},\; e_u:=u-u_\star,
$$
and $u_\star:=\frac{x_{2\star}}{1+x_{2\star}}$. A standard PD controller for the error dynamics is given by
\begin{align}\label{PDcontrol}
    e_u&=k_p e_1+k_d e_2,
\end{align}
where $k_p, k_d$ are tuning gains. Notice that, for the computation of $x_{1\star}$, the implementation of this controller requires the knowledge of $D$. The Jacobian matrix of the closed-loop system $\dot e = F(e)$, evaluated at the equilibrium point, is given by
\begin{align}
J:=\left(
               \begin{array}{cc}
                 \nabla_{e_1} F_1(e)  &  \nabla_{e_2} F_1(e)  \\
                 \nabla_{e_1} F_2(e)  &  \nabla_{e_2} F_2(e) \\
               \end{array}
             \right)\Bigg|_{e=0}
             =\left(
                \begin{array}{cc}
                  k_p(1+x_{2\star}) & k_d+k_dx_{2\star}-\frac{1}{1+x_{2\star}} \\
                  \frac{1}{1+x_{2\star}}-\frac{Dk_p(1+x_{2\star})}{x_{2\star}} & -\frac{D(-1+k_dx_{2\star}(1+x_{2\star}))}{x_{2\star}^2} \\
                \end{array}
              \right),\nonumber
\end{align}
where $\nabla_{e_i}F_j(e):={\partial F_j(e) \over \partial e_i}$. The matrix $J$ is Hurwitz if and only if its trace is negative and its determinant is positive, which are given by
\begin{align}
   \tr(J)&=k_p(1+x_{2\star})- k_d D({1 \over x_{2\star}}+1)+{D \over x_{2\star}^2}\nonumber\\
    \det(J)&=k_p\frac{D}{x_{2\star}^2}-k_d+\frac{1}{(x_{2\star}+1)^2}.\nonumber
\end{align}
Defining the positive constants
\begequarrs
m_1 &:=& {x_{2\star} \over D},\;b_1:={1 \over x_{2\star} + x_{2\star}^2}\\
m_2 &:=& {D \over x^2_{2\star}},\;b_2:={1 \over (1+ x_{2\star})^2},
\endequarrs
we can write the trace and determinant stability conditions as the two-sided inequality
\begequ
\lab{staine}
m_2 k_p + b_2> k_d  >  m_1 k_p + b_1,
\endequ
which is a conic section in the plane $k_d-k_p$, that reveals the conflicting role of the two gains. Notice that the extracted power $D$ enters in the first slope $m_1$ in the denominator, while it appears in the nominator in $m_2$---rendering harder the gain tuning task regarding this uncertain (and time-varying) parameter.

\subsection{IDA-PBC vs PD: Phase plots and transient response}
Since the system in closed-loop with the PD \eqref{PDcontrol} and  the (non-adaptive) IDA-PBC  \eqref{bucboo} lives in the plane it is possible to get the global picture of the behavior of these controllers drawing their phase plot.

\begin{figure}[H]
  \centering
  \includegraphics[scale=0.255]{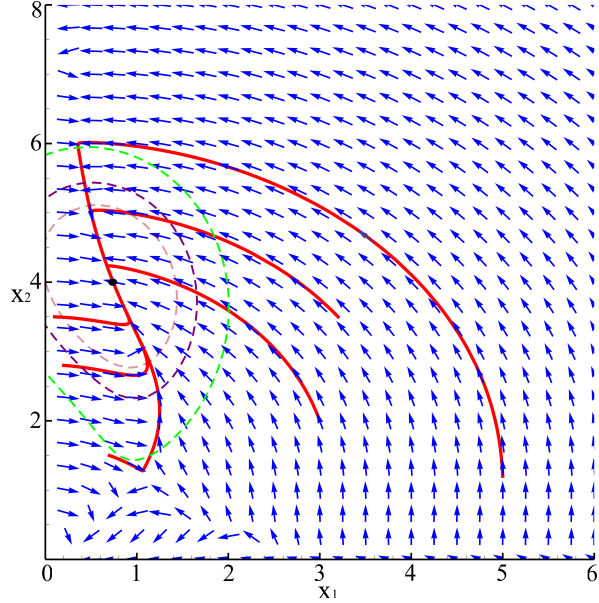}
  \caption{\normalsize Phase plot of the system with the IDA-PBC, three sublevel sets $\Omega_x$ and trajectories (red) for different initial conditions. }\label{buckboostvectorfield}
\end{figure}
In Fig. \ref{buckboostvectorfield} we show the phase plot of the IDA-PBC together with some trajectories for different initial conditions (in red), and three sublevel sets $\Omega_x$---defined in \eqref{region}---that ensure the trajectories remain in $\rea_{>0}^2$. The plot is shown for $k_1=0.01$, which satisfies \eqref{bouk1} since, for this case, $k_1'=-0.1205$ and $k_1''=-0.0583$. It is clearly seen that the state trajectories for initial conditions starting in $\Omega_x$ remain there and converge to the desired equilibrium point. Moreover, it is obvious from the phase portrait that the actual domain of attraction of this equilibrium---that ensures this key invariance property---is much larger than the one predicted by the theory. However, as shown in the plot, it cannot cover the whole positive orthant. Indeed, it is possible to show that the closed-loop vector field has another equilibrium in $\rea_{>0}^2$ that corresponds to a saddle point.  To better illustrate this fact we show in Fig. \ref{fzoo} a zoom of the phase plot around this unstable equilibrium.

In Fig. \ref{PDvectorfield} we show the phase plot for the PD controller \eqref{PDcontrol}, whose parameters are chosen as $k_p=-0.4, k_d=-1.5$ to satisfy the condition \eqref{staine}, which for the current situation takes the form
$$
0.037k_p + 0.04 > k_d  >  6.7358k_p + 0.0588.
$$
The figure shows two trajectories---that have to be chosen very close to the equilibrium---to ensure that the trajectories remain in $\rea_{>0}^2$ and converge to $x_\star$. Compared with Fig. \ref{buckboostvectorfield}, it is observed that the IDA-PBC provides a much bigger domain of attraction and, moreover, gives an estimate for it. Other values for the gains $k_p$ and $k_d$ are tried yielding similar inadmissible behavior.

\begin{figure}[H]
  \centering
  \includegraphics[scale=0.255]{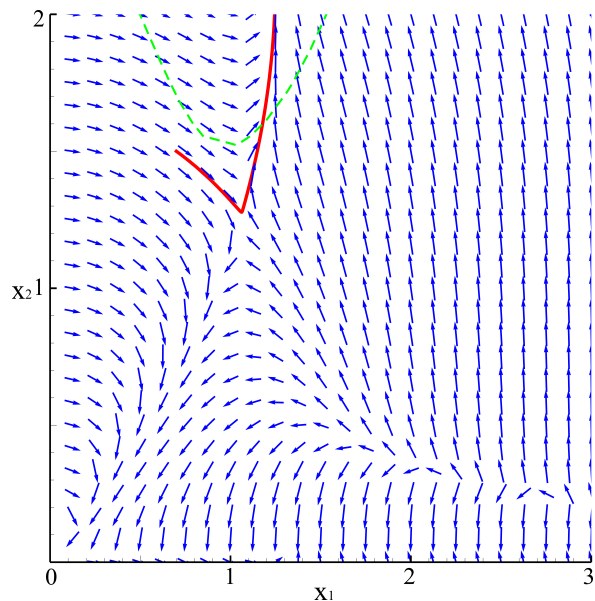}
  \caption{\normalsize Zoom of the phase plot of the system with the IDA-PBC around the saddle point.  }\label{fzoo}
\end{figure}
\begin{figure}[H]
  \centering
  \includegraphics[scale=0.255]{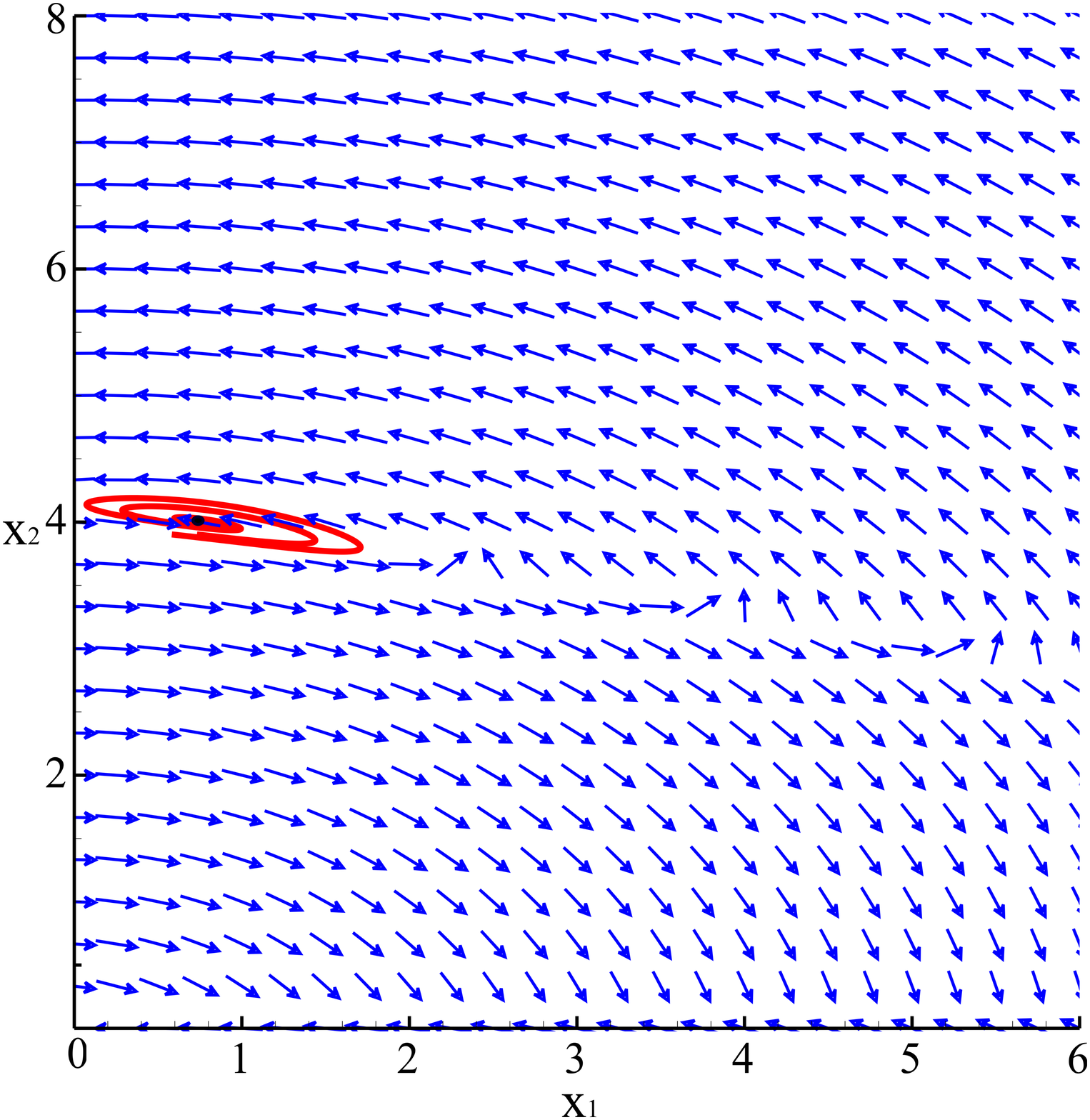}
  \caption{\normalsize Phase portrait of the system with the PD controller and trajectories (red) for two initial conditions $x(0) \approx x_\star$.  }\label{PDvectorfield}
\end{figure}

Finally, Figs. \ref{Pdx} and \ref{u1} show the transient responses of the output voltage $v$ and the duty ratio $u$ under the IDA-PBC and the PD controller with initial conditions $x(0)=(0.4,3.9)$ and gains $k_1=0.01$ for the former and $k_p=-0.4, k_d=-1.5$ for the latter. It is seen that the IDA-PBC has a faster transient performance with a smaller control signal.

\begin{figure}[H]
  \centering
  \includegraphics[scale=0.65]{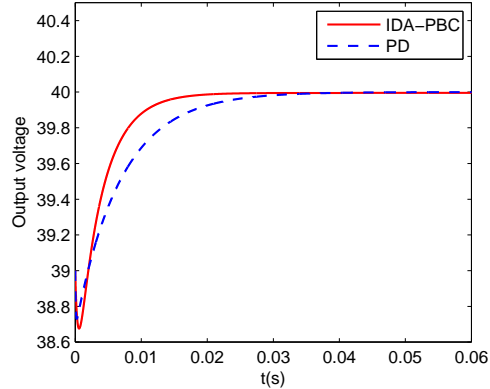}
  \caption{\normalsize Output voltage for the IDA-PBC and the PD controller.}\label{Pdx}
\end{figure}

\begin{figure}[H]
  \centering
  \includegraphics[scale=0.26]{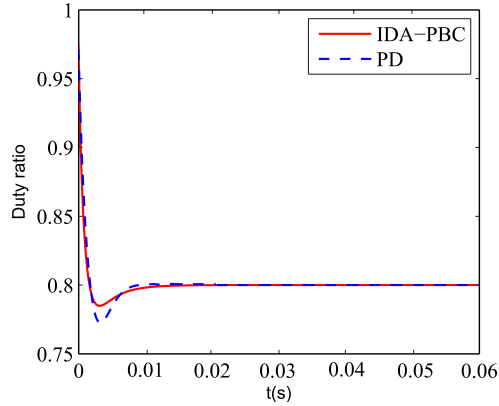}
  \caption{\normalsize Duty ratio for the IDA-PBC and the PD controller.}\label{u1}
\end{figure}

\subsection{Adaptive IDA-PBC with time-varying $D$}
Fig. \ref{controlgain} shows the profiles of output voltage and inductor current for the adaptive IDA-PBC---for different values of the control gain $k_1$ and adaptation gain $\gamma=1$---in the face of step changes in the extracted power $D$. It is seen that increasing the control gain $k_1$ reduces the convergence time of the output voltage.
As shown in the figure, the output voltage recovers very fast from the variations of the power $D$, always converging to the desired equilibrium. This is due to the fact that, as predicted by the theory, the power estimate converges---exponentially fast---to the true value independently of the control signal. It should be remarked that the PD controller becomes unstable in this scenario.

In Fig. \ref{estimator} we show the step changes in the power $D$ and the estimate $\hat D$ for different values of the adaptation gain, with the initial condition $\hat D(0)=D(0)$. As predicted by the theory, for a larger $\gamma$, the speed of convergence of the estimator is faster. Notice, however, that in the selection of $\gamma$, there is a tradeoff between convergence speed and noise sensitivity.

\begin{figure}[H]
  \centering
  \includegraphics[scale=0.65]{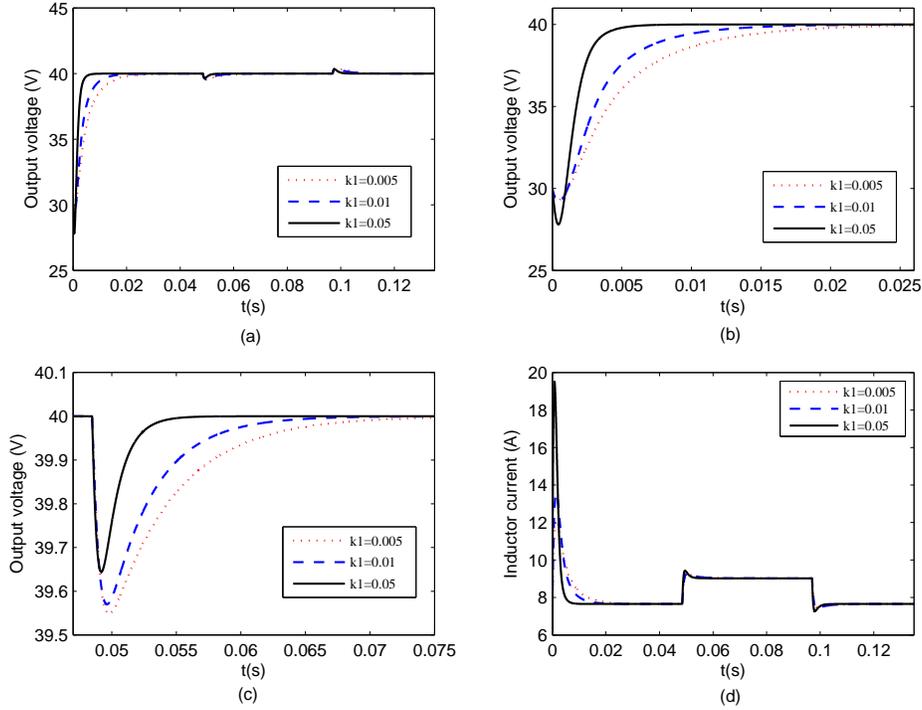}
  \caption{\normalsize Response curves for the adaptive IDA-PBC with $\gamma=1$ to changes in the power $D$: (a) output voltage---with (b) and (c) zooms for it---and (d) the inductor current.}\label{controlgain}
\end{figure}

\begin{figure}[H]
  \centering
\includegraphics[scale=0.65]{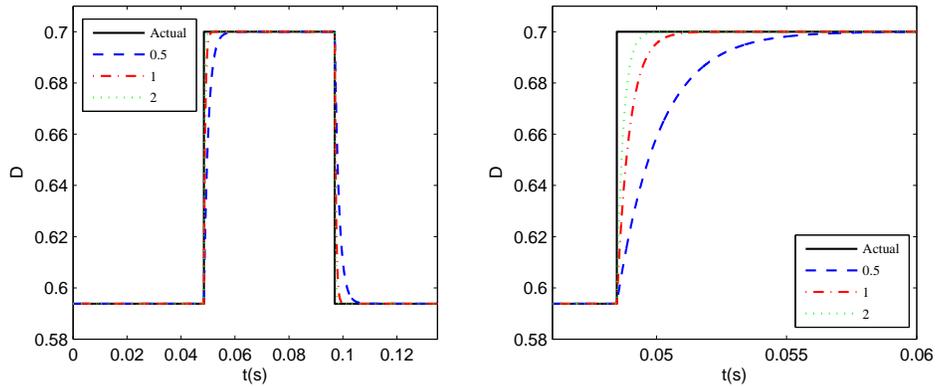}
  \caption{\normalsize Transient performance of the estimate $\hat D$ under step changes of the parameter $D$ for various adaptation gains $\gamma$ and a zoom of the first step.}\label{estimator}
\end{figure}

\section{Conclusions}\label{section6}

In this paper, we have addressed the challenging problem of regulation of the output voltage of a buck-boost converter feeding a CPL with unknown power. First, assuming the power is known, an IDA-PBC has been proposed. Subsequently, an on-line I$\&$I estimator with global convergence property has been presented to render the scheme adaptive, preserving the asymptotic stability property. We have also illustrated the performance limitations of the classical PD controller stemming from the fact that, due to the presence of the CPL,  the system is non-minimum phase. Some realistic simulations have been provided to confirm the effectiveness of the proposed method.

Although the  great complexity of  the exact expression of the controller stymies its practical application, well-established and effective methods of function approximation can be used to obtain a workable solution. Current research is under way in this direction with the final aim of reporting a practical implementation. Another line of research that we are currently pursuing is the addition of a current observer to remove the need for its measurement, which is an issue of practical interest.

\appendix

\section{Proof of Proposition \ref{proposition1}}

(P1) We will show that the control \eqref{ubucboo} can be derived using the IDA-PBC method of Proposition \ref{proposition0} with the selection
\begin{align}
F_d(x):=\left[
\begin{array}{cc}
 -\frac{x_2}{x_1} & -\frac{2x_2}{x_2+1} \\
 \frac{2x_2}{x_2+1} & -\frac{2x_1}{(x_2+1)^2}\\
 \end{array}
 \right],
\end{align}
that, for $x \in \rea_{>0}^2$, satisfies the condition \eqref{fd}.\footnote{It is well-known \citep{Ort04} that a key step for the successful application of the method is a suitable selection of this matrix, which is usually guided by the study of the solvability of the PDE \eqref{pde}. See \cite{Wei} for some guidelines for its selection in this example.}

The system \eqref{bucboo} can be rewritten in the form \eqref{dotxfg} with $g(x)$ given in \eqref{g} and the vector field
$$
f(x):=\left(
  \begin{array}{c}
    -x_2 \\
    x_1-\frac{D}{x_2} \\
  \end{array}
\right).
$$
Noting that the left annihilator of $g(x)$ is $g^\perp(x):=[x_1\;x_2+1]$, the PDE \eqref{pde} takes the form
\begin{align}
   \left[
     \begin{array}{cc}
       x_1 & x_2+1 \\
     \end{array}
   \right]
\left(\left[
    \begin{array}{c}
      -x_2 \\
      x_1-\frac{D}{x_2} \\
    \end{array}
  \right]
-\left[
   \begin{array}{cc}
 -\frac{x_2}{x_1} & -\frac{2x_2}{x_2+1} \\
 \frac{2x_2}{x_2+1} & -\frac{2x_1}{(x_2+1)^2}\\
 \end{array}
 \right]
\nabla H_d(x)
\right)=0,
\nonumber
\end{align}
which is equivalent to
\begin{align}\label{pde1}
   -x_2\nabla_{x_1} H_d(x)+2x_1\nabla_{x_2} H_d(x)=D - x_1+\frac{D}{x_2}.
\end{align}
The solution of the PDE \eqref{pde1} is easily obtained using a symbolic language, {\em e.g.}, Maple or Mathematica, and is of the form
$$
H_d(x)=-\frac{1}{2}\left(x_2+\sqrt{2}D\arctan{\left[\frac{\sqrt{2}x_1}{x_2}\right]}\right)
-\frac{D\arctan{\left[\frac{x_1}{\sqrt{x_1^2+\frac{x_2^2}{2}}}\right]}}{2\sqrt{x_1^2+\frac{x_2^2}{2}}}
+\Phi(x_1^2+\frac{x_2^2}{2}).
$$
where $\Phi(\cdot)$ is an arbitrary function. Selecting this free function as
$$
\Phi(z):={k_1 \over 2}(z+ k_2)^2,
$$
with $k_1$ and $k_2$ arbitrary constants, yields \eqref{solution}.

To complete the design it only remains to prove the existence of $k_1$ and $k_2$ that verify \eqref{mincon}. Towards this end, we first compute the gradient of $H_d(x)$ as
\begin{align}
\nabla H_d  &=\begin{bmatrix}
          -\frac{D(1+x_2)}{2x_1^2+x_2^2}+k_1x_1(2(k_2+x_1^2)+x_2^2)+\frac{\sqrt{2}Dx_1\arctan{\left[\frac{x_1}{\sqrt{x_1^2+\frac{x_2^2}{2}}}\right]}}{(2x_1^2+x_2^2)^{\frac{3}{2}}}\\
          \frac{\sqrt{2x_1^2+x_2^2}\Big(2Dx_1(1+x_2)+x_2(2x_1^2+x_2^2)\big(-1+2k_1(k_2+x_1^2)x_2+k_1x_2^3\big)\Big)+\sqrt{2}Dx_2^2\arctan{\left[\frac{x_1}{\sqrt{x_1^2+\frac{x_2^2}{2}}}\right]}}{2x_2(2x_1^2+x_2^2)^{\frac{3}{2}}}
          \end{bmatrix}\label{gradient0}
\end{align}
Evaluating \eqref{gradient0} at the equilibrium and selecting $k_2$ as given in \eqref{k2} yields
\begin{align}
\left.\nabla H_d\right|_{x=x_\star}=\begin{bmatrix}0\\
D+Dx_{2\star}-x_{1\star}x_{2\star}\end{bmatrix}.
\end{align}
Invoking \eqref{x1sta} one gets $\left.\nabla H_d\right|_{x=x_\star}=0$.

The Hessian of $H_d(x)$ is given by
\begin{align}\label{hessian}
\nabla^2 H_d=\begin{bmatrix} \nabla_{x_1}^2 H_d
&\nabla_{x_1x_2}^2 H_d \\
\nabla_{x_2x_1}^2 H_d
&\nabla_{x_2}^2 H_d\end{bmatrix},
\end{align}
where

\begin{align}
\nabla_{x_1}^2 H_d=&\frac{1}{(2x_1^2+x_2^2)^3}\bigg((2x_1^2+x_2^2)\big(2Dx_1(3+2x_2)+k_1(2x_1^2+x_2^2)^2(2k_2+6x_1^2+x_2^2)\big)\nonumber\\
&+\sqrt{2}D(-4x_1^2+x_2^2)\sqrt{2x_1^2+x_2^2}\arctan{\left[\frac{x_1}{\sqrt{x_1^2
+\frac{x_2^2}{2}}}\right]}\bigg)\nonumber\\
\nabla_{x_2x_1}^2 H_d=&\nabla_{x_1x_2}^2 H_d=\frac{1}{x_2(2x_1^2+x_2^2)^3}\bigg(2k_1x_1x_2^2(2x_1^2+x_2^2)^3+D\big(2x_1^2x_2^2-4x_1^4(1+x_2)\nonumber\\
&+x_2^4(2+x_2)\big)-3\sqrt{2}Dx_1x_2^2\sqrt{2x_1^2+x_2^2}\arctan{\frac{x_1}{\sqrt{x_1^2+\frac{x_2^2}{2}}}}\bigg)\nonumber\\
\nabla_{x_2}^2 H_d=& \frac{1}{2x_2^2(2x_1^2+x_2^2)^3}\bigg((2x_1^2+x_2^2)\Big(k_1 \big(2(k_2+x_1^2)+3x_2^2\big)(2x_1^2+x_2^3)^2-4Dx_1\big(x_1^2+\nonumber\\
&x_2^2(2+x_2)\big)\Big)+2\sqrt{2}D(x_1^2-x_2^2)x_2^2\sqrt{2x_1^2+x_2^2}\arctan{\left[\frac{x_1}{\sqrt{x_1^2+\frac{x_2^2}{2}}}\right]}\bigg).\nonumber
\end{align}
Replacing $k_2$ in \eqref{hessian} and evaluating it at the equilibrium point $x=x_\star$, we obtain
\begin{align}\label{hessian1}
\nabla^2 H_d\Bigg|_{x=x_\star}=\begin{bmatrix} \nabla_{x_1}^2 H_d|_{x=x_\star}
&\nabla_{x_1x_2}^2 H_d|_{x=x_\star}\\
\nabla_{x_2x_1}^2 H_d|_{x=x_\star}
&\nabla_{x_2}^2 H_d|_{x=x_\star}\end{bmatrix},
\end{align}
where
\begin{align}
\nabla_{x_1}^2 H_d|_{x=x_\star}=&\frac{1}{x_{1\star}(2x_{1\star}^2+x_{2\star}^2)^{\frac{5}{2}}}\Bigg(\sqrt{2x_{1\star}^2+x_{2\star}^2}\bigg(4k_1x_{1\star}^3(2x_{1\star}^2+x_{2\star}^2)^2
+D\big(x_{2\star}^2(1+\nonumber\\
&x_{2\star})+x_{1\star}^2(8+6x_{2\star})\big)\bigg)-6\sqrt{2}Dx_{1\star}^3\arctan{\frac{x_{1\star}}{\sqrt{x_{1\star}^2
+\frac{x_{2\star}^2}{2}}}}\Bigg)\nonumber\\
\nabla_{x_2x_1}^2 H_d|_{x=x_\star}=&\nabla_{x_1x_2}^2 H_d|_{x=x_\star}=\frac{1}{x_{2\star}(2x_{1\star}^2+x_{2\star}^2)^3}\Bigg(2k_1x_{1\star}x_{2\star}^2(2x_{1\star}^2+x_{2\star}^2)^3+D\big(-4x_{1\star}^4\nonumber\\
&2x_{1\star}^2x_{2\star}^2-4x_{1\star}^4x_{2\star}+x_{2\star}^4(2+x_{2\star})\big)-3\sqrt{2}Dx_{1\star}x_{2\star}^2\sqrt{2x_{1\star}^2+x_{2\star}^2}\times\nonumber\\
&\arctan{\frac{x_{1\star}}{\sqrt{x_{1\star}^2+\frac{x_{2\star}^2}{2}}}}\Bigg)\nonumber
\end{align}
\begin{align}
\nabla_{x_2}^2 H_d|_{x=x_\star}&= \frac{1}{2x_{1\star}x_{2\star}^2(2x_{1\star}^2+x_{2\star}^2)^{\frac{5}{2}}}\Bigg(\sqrt{2x_{1\star}^2
+x_{2\star}^2}\bigg(2k_1x_{1\star}x_{2\star}^4(2x_{1\star}^2+x_{2\star}^2)^2+D\big(x_{2\star}^4\nonumber\\
&-4x_{1\star}^4+x_{2\star}^5)-2x_{1\star}^2x_{2\star}^2(3+x_{2\star})\big)\bigg)-3\sqrt{2}Dx_{1\star}x_{2\star}^4\arctan{\frac{x_{1\star}}{\sqrt{x_{1\star}^2+\frac{x_{2\star}^2}{2}}}}\Bigg).\nonumber
\end{align}
Some lengthy, but straightforward, calculations prove that $\nabla_{x_1}^2 H_d|_{x=x_\star}>0$ holds if and only if
\begequ
\lab{k1pri}
k_1 > k_1',
\endequ
where $k_1'$ is defined as
\begin{align}
k_1'&:=\frac{\frac{6\sqrt{2}Dx_{1\star}\arctan{\frac{x_{1\star}}{\sqrt{x_{1\star}^2+\frac{1}{2}x_{2\star}^2}}}}{\sqrt{2x_{1\star}^2+x_{2\star}^2}}-D \big(x_{2\star}^2(1+x_{2\star})+ x_{1\star}^2(8+6x_{2\star})\big)}{4x_{1\star}^3(2x_{1\star}^2+x_{2\star}^2)^2}.
\lab{kpri}
\end{align}
Moreover, the determinant of \eqref{hessian1} is given by
\begin{align}
\det\left(\left.\nabla^2 H_d\right|_{x=x_\star}\right)=&\frac{1}{2x_{1\star}^2x_{2\star}^2(2x_{1\star}^2+x_{2\star}^2)^{4}}\Bigg(\Big(4k_1x_{1\star}^3(2x_1^2+x_2^2)^2
+D\big(x_{2\star}^2(1+x_{2\star})+\nonumber\\
&x_{1\star}^2(8+6x_{2\star})\big)-\frac{6\sqrt{2}Dx_{1\star}^3\arctan{\frac{x_{1\star}}{\sqrt{x_{1\star}^2
+\frac{x_{2\star}^2}{2}}}}}{\sqrt{2x_{1\star}^2+x_{2\star}^2}}\Big)
\Big(2k_1x_{1\star}x_{2\star}^4(2x_{1\star}^2+\nonumber\\
&x_{2\star}^2)^2-\frac{3\sqrt{2}Dx_{1\star}x_{2\star}^4\arctan{\frac{x_{1\star}}{\sqrt{x_{1\star}^2+\frac{x_{2\star}^2}{2}}}}}{\sqrt{2x_{1\star}^2
+x_{2\star}^2}}+D\big(x_{2\star}^4(1+x_{2\star})-4x_{1\star}^4-\nonumber\\
&2x_{1\star}^2x_{2\star}^2(3+x_{2\star})\big)\Big)\Bigg)-\frac{1}{x_{2\star}^2(2x_{1\star}^2+x_{2\star}^2)^{4}} \Bigg(2k_1x_{1\star}x_{2\star}^2(2x_{1\star}^2+x_{2\star}^2)^2+\nonumber\\
&D\big(x_{2\star}^2(2+x_{2\star})-x_{1\star}^2(2+2x_{2\star})\big)-\frac{3\sqrt{2}Dx_1x_2^2\arctan{\frac{x_1}{\sqrt{x_{1\star}^2
+\frac{x_{2\star}^2}{2}}}}}{\sqrt{2x_1^2+x_2^2}}\Bigg)^2\nonumber\\
=&\frac{1}{2x_{1\star}^2x_{2\star}^2(2x_{1\star}^2+x_{2\star}^2)^{4}}\Bigg(\bigg(4k_1x_{1\star}^3(2x_{1\star}^2+x_{2\star}^2)^2
+D\big(x_{2\star}^2(1+x_{2\star})+\nonumber\\
&x_{1\star}^2(8+6x_{2\star})\big)-\frac{6\sqrt{2}Dx_{1\star}^3\arctan{\frac{x_{1\star}}{\sqrt{x_{1\star}^2
+\frac{x_{2\star}^2}{2}}}}}{\sqrt{2x_{1\star}^2+x_{2\star}^2}}\bigg)
\bigg(2k_1x_{1\star}x_{2\star}^4(2x_{1\star}^2+\nonumber\\
&x_{2\star}^2)^2+D\big(-4x_{1\star}^4+x_{2\star}^4(1+x_{2\star})-2x_{1\star}^2x_{2\star}^2(3+x_{2\star})\big)-\nonumber\\
&\frac{3\sqrt{2}Dx_{1\star}x_{2\star}^4\arctan{\frac{x_{1\star}}{\sqrt{x_{1\star}^2+\frac{x_{2\star}^2}{2}}}}}{\sqrt{2x_{1\star}^2
+x_{2\star}^2}}\bigg)-2x_{1\star}^2\bigg(2k_1x_{1\star}x_{2\star}^2(2x_{1\star}^2+x_{2\star}^2)^2+\nonumber\\
&D\big(x_{2\star}^2(2+x_{2\star})-x_{1\star}^2(2+2x_{2\star})\big)-\frac{3\sqrt{2}Dx_1x_2^2\arctan{\frac{x_1}{\sqrt{x_{1\star}^2
+\frac{x_{2\star}^2}{2}}}}}{\sqrt{2x_{1\star}^2+x_{2\star}^2}}\bigg)^2\Bigg)\nonumber\\
=&\frac{1}{2x_{1\star}^2x_{2\star}^2(2x_{1\star}^2+x_{2\star}^2)^{4}}\Bigg(k_1D(2x_{1\star}^2+x_{2\star}^2)^2 h(x_{\star}) +2D^2 \big(x_{2\star}^2(1+x_{2\star})+\nonumber\\
&x_{1\star}^2(8+6x_{2\star})\big)\big(x_{2\star}^4(1+x_{2\star})-4x_{1\star}^4
-x_{1\star}^2x_{2\star}^2(3+x_{2\star})\big)-2x_{1\star}^2\big(x_{2\star}^2\nonumber\\
&(2+x_{2\star})-2x_{1\star}^2(2+2x_{2\star})\big)^2-3\sqrt{2}x_{1\star}x_{2\star}^4\big(x_{2\star}^2(1+x_{2\star})
+x_{1\star}^2(8+\nonumber\\
&6x_{2\star})\big)\frac{\arctan{\frac{x_{1\star}}{\sqrt{x_{1\star}^2+\frac{x_{2\star}^2}{2}}}}}{\sqrt{2x_{1\star}^2+x_{2\star}^2}}-6\sqrt{2}Dx_{1\star}^3\big(-4x_{1\star}^4
+x_{2\star}^4(1+x_{2\star})\nonumber\\
&-2x_{1\star}^2x_{2\star}^2(3+x_{2\star})\big)\frac{\arctan{\frac{x_{1\star}}{\sqrt{x_{1\star}^2+\frac{x_{2\star}^2}{2}}}}}{\sqrt{2x_{1\star}^2+x_{2\star}^2}}\Bigg)\nonumber
\end{align}
where $h(x_{\star})$ is defined as
\begequ
\lab{h}
h(x_\star):= 4x_{1\star}^3+4x_{1\star}^5x_{2\star}^3+2x_{1\star}^3x_{2\star}^4+x_{1\star}x_{2\star}^2+x_{1\star}x_{2\star}^7-8x_{1\star}^7-4x_{1\star}^5x_{2\star}^2.
\endequ
 Consequently, $\det\left(\left.\nabla^2 H\right|_{x=x_\star}\right)>0$ holds provided
\begin{align}
&k_1D(2x_{1\star}^2+x_{2\star}^2)^2 h(x_{\star}) +\nonumber\\
&2D^2\big(x_{2\star}^2(1+x_{2\star})+x_{1\star}^2(8+6x_{2\star})\big)\big(x_{2\star}^4(1+x_{2\star})-4x_{1\star}^4
-x_{1\star}^2x_{2\star}^2(3+x_{2\star})\big)-2\big(x_{2\star}^2(2+\nonumber\\
&x_{2\star})-2x_{1\star}^2(2+2x_{2\star})\big)^2-3\sqrt{2}x_{1\star}x_{2\star}^4\big(x_{2\star}^2(1+x_{2\star})+x_{1\star}^2(8+6x_{2\star})\big)
\frac{\arctan{\frac{x_{1\star}}{\sqrt{x_{1\star}^2+\frac{x_{2\star}^2}{2}}}}}{\sqrt{2x_{1\star}^2+x_{2\star}^2}}-\nonumber\\
&6\sqrt{2}Dx_{1\star}^3\big(-4x_{1\star}^4+x_{2\star}^4(1+x_{2\star})-2x_{1\star}^2x_{2\star}^2(3+x_{2\star})\big)\frac{\arctan{\frac{x_{1\star}}{\sqrt{x_{1\star}^2
+\frac{x_{2\star}^2}{2}}}}}{\sqrt{2x_{1\star}^2+x_{2\star}^2}}>0.\lab{keyine}
\end{align}
Our final task is to select $k_1$, which stands in the first left hand term above, to ensure \eqref{keyine} holds.

First, we notice that $h(x_\star)$ may be factored as
$$
h(x_\star)= 4x_{1\star}^3+2x_{1\star}^3x_{2\star}^4+x_{1\star}x_{2\star}^2+x_{1\star}x_{2\star}^7+4x_{1\star}^5[x_{2\star}^3 -(2x_{1\star}^2+x_{2\star}^2)],
$$
with the term in brackets being positive in most practical applications. Consequently, since all other terms are positive, we have that $h(x_{\star})>0$ and the inequality \eqref{keyine} is equivalent to $k_1>k_1''$, with the latter defined as
\begin{align}
k_1''&:= -\frac{1}{D(2x_{1\star}^2+x_{2\star}^2)^2 h(x_\star)} \times \Bigg(2D^2 \big(x_{2\star}^2(1+x_{2\star})+x_{1\star}^2(8+6x_{2\star})\big)
\big(-4x_{1\star}^4+x_{2\star}^4(1+\nonumber\\
&x_{2\star})-x_{1\star}^2x_{2\star}^2(3+x_{2\star})\big)-2x_{1\star}^2\big(x_{2\star}^2(2+x_{2\star})
-2x_{1\star}^2(2+2x_{2\star})\big)^2-3\sqrt{2}x_{1\star}x_{2\star}^4\big(x_{2\star}^2(1+\nonumber\\
&x_{2\star})
+x_{1\star}^2(8+6x_{2\star})\big)\frac{\arctan{\frac{x_{1\star}}{\sqrt{x_{1\star}^2+\frac{x_{2\star}^2}{2}}}}}{\sqrt{2x_{1\star}^2+x_{2\star}^2}}-6\sqrt{2}Dx_{1\star}^3\big(-4x_{1\star}^4
+x_{2\star}^4(1+x_{2\star})\nonumber\\
&-2x_{1\star}^2x_{2\star}^2(3+x_{2\star})\big)\frac{\arctan{\frac{x_{1\star}}{\sqrt{x_{1\star}^2
+\frac{x_{2\star}^2}{2}}}}}{\sqrt{2x_{1\star}^2+x_{2\star}^2}}\Bigg).\lab{kbipri}
\end{align}
Combining this constraint with \eqref{k1pri} allows us to conclude that $k_1> \max {\{k_1', k_1''\}}$ ensures $\nabla^2 H|_{x=x_\star}>0$.

The proof of P1 is completed showing that the IDA-PBC \eqref{ubucboo} results replacing the data in \eqref{uida}.

(P2) The proof of this claim follows immediately noting that we have shown above that the function $H_d(x)$ has a positive definite Hessian evaluated at $x_\star$, therefore it is {\em convex}. Consequently, for sufficiently small $c$, the sublevel set $\Omega_x$ defined in \eqref{region} is bounded and strictly contained in $\mathbb{R}^2_{>0}$. The proof is completed recalling  that sublevel sets of strict Lyapunov functions are inside the domain of attraction of the equilibrium.

\qed

\end{document}